\documentclass[12pt]{elsarticle}
\usepackage{amsmath,amsfonts}
\usepackage{subfigure}
\usepackage{algorithm,  algpseudocode}
\usepackage{graphicx}
\usepackage{kotex}
\usepackage{array}
\usepackage{tabularx}

\usepackage{makecell}
\usepackage{multirow}
\usepackage{stfloats}

\newtheorem{corollary}{\bfseries Corollary}

\newtheorem{lemma}{\bfseries Lemma}

\newtheorem{remark}{\bfseries Remark}
\newtheorem{theorem}{\bfseries Theorem}

\newtheorem{example}{\bfseries Example}
\newenvironment{proof}{\paragraph{Proof:}}{\hfill$\square$}

\usepackage{amsmath, amssymb, mathtools, accents, amsfonts}
\usepackage{algorithm, algorithmicx, algpseudocode}
\usepackage{subfigure}
\usepackage{mathtools}
\usepackage{blkarray, bigstrut}
\usepackage{booktabs}
\usepackage{bbm}
\usepackage{tikz}
	\usetikzlibrary{shapes,arrows}
	\tikzstyle{frame} = [draw, -latex]
	\tikzstyle{line} = [draw]
	\tikzstyle{line2} = [draw, dashdotted]
	\tikzstyle{line3} = [draw, dashed]
	\tikzstyle{line3UD} = [draw, dashed]
	\tikzstyle{place} = [circle, draw=black, fill=white, thick, inner sep=2pt, minimum size=1mm]
	\tikzstyle{place2} = [circle, draw=black, fill=black, thick, inner sep=2pt, minimum size=1mm]
	\tikzstyle{placeRed} = [circle, draw=red, fill=red, thick, inner sep=2pt, minimum size=1mm]
	\tikzstyle{vertex} = [circle, draw=black, fill=black, thick, inner sep=2pt, minimum size=1mm]
\usepackage{tkz-euclide}
\usetikzlibrary{shapes,arrows} 
\def\BState{\State\hskip-\ALG@thistlm}
\algdef{SE}[DOWHILE]{Do}{doWhile}{\algorithmicdo}[1]{\algorithmicwhile\ #1}%

\tikzstyle{decision} = [diamond, draw, fill=blue!20,
    text width=4.5em, text badly centered, node distance=3cm, inner sep=0pt]
\tikzstyle{block1} = [rectangle, draw, text width=8em, text centered, minimum height=4em]
\tikzstyle{block2} = [rectangle, draw, text width=3em, text centered, minimum height=4em]
\tikzstyle{block3} = [rectangle, draw, text width=11em, text centered, minimum height=12em, dashed,black]
\tikzstyle{block4} = [rectangle, draw, text width=11em, text centered, minimum height=18em, dashed,black]
\tikzstyle{block5} = [rectangle, draw, text width=11em, text centered, minimum height=32em, dashed,black]
\tikzstyle{block6} = [rectangle, draw, text width=11em, text centered, minimum height=18.5em, dashed,black]
\tikzstyle{block7} = [rectangle, draw, text width=11em, text centered, minimum height=11.8em, dashed,black]
\tikzstyle{line01} = [draw, -latex']
\tikzstyle{line02} = [draw, latex'-latex']

\usepackage{setspace}
\usepackage[pagewise]{lineno}
\newcommand*\patchAmsMathEnvironmentForLineno[1]{%
\expandafter\let\csname old#1\expandafter\endcsname\csname #1\endcsname
\expandafter\let\csname oldend#1\expandafter\endcsname\csname end#1\endcsname
\renewenvironment{#1}%
{\linenomath\csname old#1\endcsname}%
{\csname oldend#1\endcsname\endlinenomath}}%
\newcommand*\patchBothAmsMathEnvironmentsForLineno[1]{%
\patchAmsMathEnvironmentForLineno{#1}%
\patchAmsMathEnvironmentForLineno{#1*}}%
\AtBeginDocument{%
\patchBothAmsMathEnvironmentsForLineno{equation}%
\patchBothAmsMathEnvironmentsForLineno{align}%
\patchBothAmsMathEnvironmentsForLineno{flalign}%
\patchBothAmsMathEnvironmentsForLineno{alignat}%
\patchBothAmsMathEnvironmentsForLineno{gather}%
\patchBothAmsMathEnvironmentsForLineno{multline}%
}

\journal{Journal of the Franklin Institute}

\begin{document}
\let\WriteBookmarks\relax
\def\floatpagepagefraction{1}
\def\textpagefraction{.001}

\begin{frontmatter}



\title{Fixed Node Determination and Analysis in Directed Acyclic Graphs of Structured Networks}


\author[1]{Nam-jin Park}
\ead{namjinpark@gist.ac.kr}

\author[1]{Yeong-Ung Kim}
\ead{yeongungkim@gm.gist.ac.kr}

\author[1]{Hyo-Sung Ahn\corref{cor1}}
\ead{hyosung@gist.ac.kr}

\affiliation[1]{organization={Gwangju Institute of Science and Technology (GIST)},
            city={Gwangju},
            country={Republic of Korea}}

\cortext[cor1]{Corresponding author}

\begin{abstract}
This paper explores the conditions for determining fixed nodes in structured networks, specifically focusing on directed acyclic graphs (DAGs). 
We introduce several necessary and sufficient conditions for determining fixed nodes in $p$-layered DAGs.
This is accomplished by defining the problem of maximum disjoint stems, based on the observation that all DAGs can be represented as hierarchical structures with a unique label for each layer.
For structured networks, we discuss the importance of fixed nodes by considering their controllability against the variations of network parameters.
Moreover, we present an efficient algorithm that simultaneously performs labeling and fixed node search for $p$-layered DAGs with an analysis of its time complexity. 
The results presented in this paper have implications for the analysis of controllability at the individual node level in structured networks.
\end{abstract}



\begin{keyword}
Network controllability \sep fixed controllable subspace\sep fixed nodes\sep directed acyclic graphs

\end{keyword}

\end{frontmatter}

\section{Introduction}
\subsection{Controllability of Structured Networks}
The problem of network controllability is a research topic that has gained considerable attention in recent years, particularly in the context of structured networks.
In a structured network, edge weights are divided into zero or non-zero, and the study of their controllability is known as \textit{structural controllability}, which was first introduced by Lin \cite{lin1974structural}.
A structured network is called structurally controllable if the network is controllable for almost all network parameters, i.e., choices of non-zero edge weights \cite{lu2020sampled}.
On the other hand, the network controllability that considers all choices of edge weights is being studied under the name of \textit{strong structural controllability} \cite{jarczyk2011strong,mousavi2017structural,srighakollapu2021strong,zhu2022strong}
with various approaches such as PMI sequences \cite{yaziciouglu2016graph}, zero forcing sets \cite{mousavi2017structural}, 
and maximum matching \cite{shabbir2019computation}. 
However, due to the \textit{generic property} \cite{dion2003generic} of structural controllability, this field is more actively researched than strong structural controllability. 
Although a structurally controllable network may be uncontrollable for specific network parameters, its generic property makes it controllable for the majority of network parameters. 
As a result, structural controllability is more commonly studied in real-world networks compared to strong structural controllability.

On the other hand, when a structured network is not structurally controllable, a controllable subnetwork \cite{iudice2019node} can be analyzed through the \textit{controllable subspace}.
The dimension of controllable subspace, which is equivalent to the rank of the controllability matrix, is determined by the network parameters.
The upper bound of the dimension of controllable subspace is known as the \textit{generic dimension of controllable subspace} and has been studied in the framework of structural controllability \cite{zhang2013upper,hosoe1980determination,jarczyk2010determination}.
While the dimension of controllable subspace provides some information, it does not offer a complete insight into the controllability of individual nodes within the structured networks.

\subsection{Controllability of Individual Nodes}

Determining the controllability of individual nodes depends on the variations of network parameters is an important problem in structured networks.
To fully understand the controllability of individual nodes in structured networks, we need to consider two types of control uncertainties induced by network parameters:
The first type of control uncertainty arises from the algebraic constraints of the network parameters.
The second type of control uncertainty originates from variations of network parameters that lead to the non-uniqueness of controllable subspaces.
Addressing the first type of control uncertainty can be achieved using the notion of strong structural controllability, whereas the second type requires a different approach.
\textcolor{black}{To counter this, the idea of a \textit{fixed controllable subspace} has been introduced in \cite{commault2017fixed}. 
This concept represents the fixed part of controllable subspaces across different parameter variations.
Moreover, \cite{commault2017fixed} offers conditions for individual nodes included in the fixed controllable subspace. 
Specifically, a node is referred to as a \textit{fixed node} if its corresponding standard basis is included in the fixed controllable subspace, which will be defined in detail in later sections.}

This analysis is essential as it identifies individual nodes that remain controllable regardless of the variation in network parameters, providing a robust perspective on structured networks.
\textcolor{black}{
Expanding on this foundational work, the authors in \cite{van2019dynamic} propose conditions for non-fixed nodes using the concept of the \textit{supremal minimal separator}. 
However, the conditions for fixed nodes proposed in \cite{van2019dynamic} provide only the necessary conditions. 
Therefore, the complete characterization of fixed nodes through necessary and sufficient graph-theoretical conditions continues to be an open problem.} 
Exploring the conditions for the fixed nodes in various types of graphs is important.
In particular, DAGs are widely used in various fields, such as network controllability \cite{liu2012control,czeizler2018structural}, 
routing \cite{hua2008optimal}, scheduling \cite{chen2019timing,arif2020hybrid}, and social networks \cite{chen2014exploiting}.
DAGs allow clear and non-circular data flow, which is especially suitable for representing cause-and-effect relationships or temporal sequences.
Moreover, due to the non-cyclical nature of DAGs, they offer significant characteristics over conventional directed cyclic graphs, 
particularly when it comes to the computational efficiency of determining graph-theoretical conditions. This efficiency greatly aids in the analysis of network controllability.

\subsection{Research Flow}
This paper introduces the notion of fixed nodes in structured networks from a graph perspective.
We first discuss the control uncertainties that can exist in structured networks and their implications on network controllability.
Based on the notion of fixed controllable subspace introduced in \cite{commault2017fixed}, 
we interpret the properties of fixed nodes in DAGs by defining the problem of maximum disjoint stems.
Based on the fact that DAGs can always be represented as hierarchical structures with unique layers, 
we present several intuitive conditions under which individual nodes in each layer of a DAG can be fixed nodes.
Then, we show that finding the maximum disjoint stems with $m$-leaders in DAG is equivalent to the well-known $m$-vertex disjoint path problem ($k$-VDPP) \cite{robertson1995graph}.  
Finally, we introduce an efficient algorithm that simultaneously performs labeling and fixed node search for DAGs, and an analysis of its time complexity is also provided.

\subsection{Contributions}
This paper provides a graph-theoretical analysis of the controllability of individual nodes within structured networks, considering the control uncertainties they inherently possess. This perspective enhances our understanding and insights into the controllability characteristics of structured networks. The main contributions of this paper are as follows.
\begin{itemize}

\item 
\textcolor{black}{
In contrast to \cite{van2019dynamic}, which presents only necessary conditions for fixed nodes,
this paper introduces necessary and sufficient conditions for fixed nodes in DAGs with the problem of maximum disjoint stems (\textit{Theorem~\ref{thm1}} and \textit{Theorem~\ref{thm3}}). 
This provides insights into the graph-theoretical properties of fixed nodes in DAGs.}

\item 
\textcolor{black}{
Based on the necessary and sufficient conditions for fixed nodes, 
the paper presents an efficient algorithm that simultaneously performs labeling and fixed node search for $p$-layered DAGs (\textit{Algorithm~\ref{alg}}), along with an analysis of its time complexity. 
This offers a practical approach to understanding network controllability in DAGs.}

\end{itemize}
Overall, the results presented in this paper contribute to the analysis of controllability at the individual node level in structured networks, particularly in DAGs.
This paper is organized as follows.
In \textit{Section~\ref{sec_pre}}, we provide preliminaries and the existing results on network controllability from the perspective of control uncertainty.
In \textit{Section~\ref{sec_fix}}, we introduce the hierarchical structure of DAGs with a labeling algorithm and provide several conditions for fixed nodes in layered DAGs.
In \textit{Section~\ref{sec_alg}}, we present an efficient algorithm to determine fixed nodes in DAGs and provide its time complexity analysis.
The conclusions and topological examples are given in \textit{Section~\ref{sec_ex}} and \textit{Section~\ref{sec_con}}, respectively.

\section{Preliminaries}\label{sec_pre}
\subsection{State space representation}
Let us consider a network represented by:
\begin{equation} \label{eq:2}
\dot{x}=Ax+Bu,
\end{equation}
\textcolor{black}{where $x=[x_1,...,x_n]\in\mathbb{R}^{n}$, $u=[u_1,...,u_m]\in\mathbb{R}^{m}$}, and $A\in\mathbb{R}^{n\times n}$ is the adjacency matrix including the constant connection weights $[A]_{i,j}$ from states $x_i$ to $x_j$ for $i,j\in\{1,...,n\}$.
The input matrix $B \in \mathbb{R}^{n \times m}$ represents the input-connection from $m$-inputs to $n$-states.
Note that input can only be connected to one state, that is, one-to-one correspondence.
\textcolor{black}{
For a given matrix $A \in \mathbb{R}^{n \times n}$, we define a family set $Q(A)$ comprising all matrices that share the same non-zero/zero patterns with $A$ as:}
\begin{equation} \label{fam}
\textcolor{black}{Q(A)=\{A_\Lambda:[A_\Lambda]_{i,j}=a_{ij}(\neq 0) \Longleftrightarrow [A]_{i,j}\neq 0\},}
\end{equation}
\textcolor{black}{where $\Lambda$ represents specific connection weights within the structured network.}
Then, the structured network of \eqref{eq:2} can be represented as:
\begin{equation} \label{eq:2_1}
\dot{x}=A_{\Lambda}x+B_{\Lambda}u,
\end{equation}
\textcolor{black}{where $A_{\Lambda}\in Q(A)$ and $B_{\Lambda} \in Q(B)$.}
Note that changing the non-zero weights of the input matrix $B_{\Lambda}$ does not affect the controllability of a network. 
This paper assumes that the non-zero weights in $B_{\Lambda}$ are fixed to 1. 
A structured network given by \eqref{eq:2_1} is considered structurally controllable if $\mathcal{C}_{\Lambda}$ has full row rank for almost all network parameters,
where $\mathcal{C}_{\Lambda}$ is given as:
\begin{equation} \label{eq:3}
\mathcal{C}_{\Lambda}=[B_{\Lambda}, A_{\Lambda}B_{\Lambda}, A^2_{\Lambda}B_{\Lambda}, ... , A^{n-1}_{\Lambda}B_{\Lambda}].
\end{equation}

\textcolor{black}{The \textit{controllable subspace} of \eqref{eq:2_1} is defined by the column space of the controllability matrix in  \eqref{eq:3}.
Further, the maximum rank of the controllability matrix, given by \eqref{eq:3}, is designated as the \textit{generic dimension of controllable subspace} of a structured network.
Note that this controllable subspace with the maximum rank may also vary depending on network parameters, specifically, the non-zero weights in $A_{\Lambda}$. 
If the structured network given by \eqref{eq:2_1} is uncontrollable, a primary concern is to determine if each state can independently reach a desired point.}
Note that there exist certain states, known as dependent states, that can become controllable by sacrificing the controllability of other states that they are dependent on in the network.


\begin{figure*}[t]
    \centering
    \subfigure[$a_{12}\gg a_{32}$]{
        {\includegraphics[width=1.71in,height=3.7in,clip,keepaspectratio]{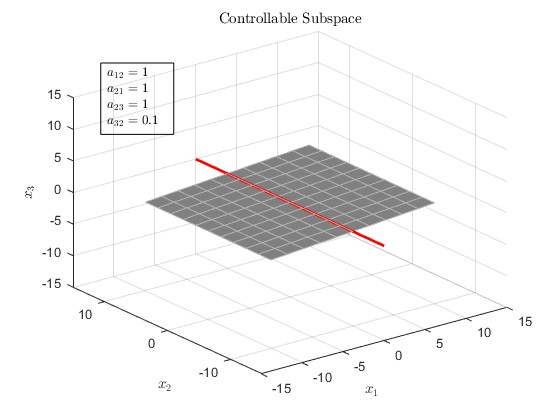}}}
    \subfigure[$a_{12}=a_{32}$]{
        {\includegraphics[width=1.71in,height=4in,clip,keepaspectratio]{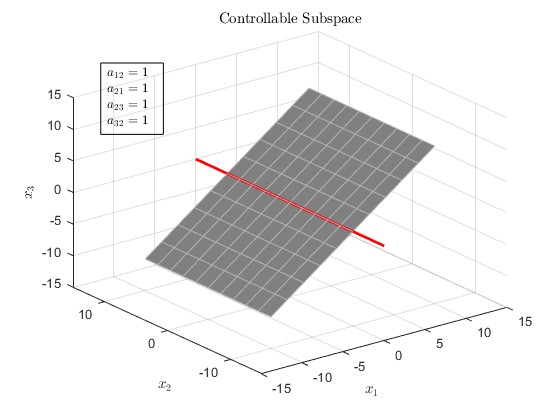}}}
    \subfigure[$a_{12}\ll a_{32}$]{
        {\includegraphics[width=1.71in,height=4in,clip,keepaspectratio]{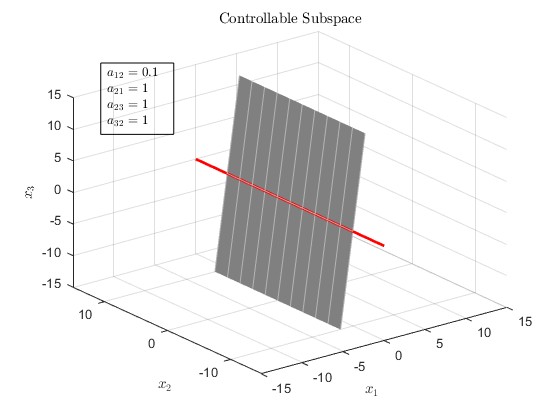}}}
    \caption{The controllable subspace of Equation \eqref{cont_ex1} varies by the network parameters.
    However, here we illustrate the most characteristic configurations under varying network parameters. The $x_2$-axis is represented by the red line.
    (a) The controllable subspace nearly aligns with the $x_1x_2$-plane when $a_{12}$ is dominant over $a_{32}$. 
    (b)The controllable subspace when $a_{12}$ is equal to $a_{32}$.
    (c) The controllable subspace almost aligns with the $x_2x_3$-plane when $a_{32}$ is dominant over $a_{12}$.
    }   
    \label{fig_CS}
\end{figure*}

\begin{example}
Let us consider $A_{\Lambda}$ and $B_{\Lambda}$ as:
\begin{align} \label{ex_cont2}
\begin{split}
&A_{\Lambda}=
\begin{bmatrix}
0 & a_{12} & 0 \\
a_{21} & 0 & a_{23} \\
0 & a_{32} & 0 \\
\end{bmatrix}, \,\,
B_{\Lambda}=
\begin{bmatrix}
0 \\
1 \\
0
\end{bmatrix}.
\end{split}
\end{align}

The controllability matrix of \eqref{ex_cont2} is given by:
\begin{align}\label{cont_ex1}
\begin{split}
\mathcal{C}_{\Lambda}=
\begin{bmatrix}
0 & a_{12} & 0 \\
1 & 0        & a_{12}a_{21}+a_{23}a_{32} \\
0 & a_{32} & 0 \\
\end{bmatrix},
\end{split}
\end{align}\textcolor{black}{
where the rank of $\mathcal{C}_{\Lambda}$ is 2, it follows that the controllable subspace, i.e., the column space of \eqref{cont_ex1}, is a two-dimensional space.}
Furthermore, the second row is linearly independent of the first and third rows,
which means that either the state pair $x_1,x_2$ or $x_2,x_3$ can be controllable, but the states $x_1$ and $x_3$ cannot be controlled simultaneously.
It implies that $x_2$ is controllable independently, and $x_1$ and $x_3$ are dependent on each other.
\end{example}

The above example demonstrates that the controllability of certain states may be uncertain depending on the variations of network parameters.
\textcolor{black}{
For structured networks, a state that can reach a desired point regardless of network parameters is termed a \textit{controllable state}.
According to \cite{commault2017fixed}, for a state to be considered a controllable state, its corresponding standard basis vector must be included in all feasible controllable subspaces of \eqref{eq:2_1}.
To determine controllable states, the authors in \cite{commault2017fixed} introduced the notion of a \textit{fixed controllable subspace}, 
which represents the intersection of all controllable subspaces with the same generic dimension.
Let a standard column basis corresponding to the state $x_i$ be denoted as $e_i$, where $e_i\in \mathbb{R}^n$ is an $n$-dimensional vector with only one non-zero value at $i$-th element.
For the structured network given by \eqref{eq:2_1}, we say that a state $x_i\in \mathbb{R}^1,i\in\{1,...,n\}$, is a controllable state
if its corresponding standard basis $e_i$ is in the fixed controllable subspace for $i\in\{1,...,n\}$.}
To facilitate a better understanding, let us consider the controllable subspaces of \eqref{cont_ex1} by the variation of network parameters shown in Fig.~\ref{fig_CS}.
As illustrated, the $x_1$ and $x_3$ axes are linearly dependent on each other and together form a plane within the controllable subspace.
Interestingly, the $x_2$-axis, indicated by the red line in Fig.~\ref{fig_CS}, remains within the controllable subspace regardless of network parameter variations. 
\textcolor{black}{Hence, the fixed controllable subspace of \eqref{cont_ex1} is composed only of the $x_2$-axis.
It follows that the corresponding standard basis $e_2\in\mathbb{R}^{3}$ is included in the fixed controllable subspace. 
In this case, we say that the state $x_2$ is a controllable state.}
The concept of a fixed controllable subspace allows us to determine controllable states independently of variations in network parameters.
This approach aids in identifying states that can be independently controlled, offering a more intuitive understanding of the controllability properties within structured networks.

\subsection{Graph theory representation}
In the previous subsection, we introduced the controllable subspace of structured networks from a state space point of view. In this section, 
we explore the controllable subspace of structured networks from a graph theory perspective.
The structured network given by \eqref{eq:2_1} can be represented as a graph with the set of state nodes and the set of edges:
\begin{equation} \label{eq:4}
\mathcal{G}(\mathcal{V},\mathcal{E}),
\end{equation}
where $\mathcal{V}$ is a set of state nodes, and $\mathcal{E}$ is the set of edges, which includes the non-zero/zero connection between the state nodes in $\mathcal{V}$.
A state node $k\in\mathcal{V}$ with an input $u$ with an edge $(u,k)\in\mathcal{E}$ is referred to as a \textit{leader}. 
We denote the set of leaders in $\mathcal{G}$ as $\mathcal{V}_{\mathcal{L}}\subset\mathcal{V}$.
The set of in-neighboring nodes of $i$ is symbolically written as $\mathcal{N}_{i}$, that is, 
if the edge weight $a_{ij}$ is non-zero, then there exists an edge $(j,i)\in\mathcal{E}$ and $j\in\mathcal{N}_i$.
The \textit{path} in $\mathcal{G}$ from a state node $i_1$ to a state node $i_p$ 
is a sequence of the edges $(i_1,i_2), (i_2,i_3), ... , (i_{p-1},i_p)$ such that 
$i_k\in\mathcal{V}$ for $k\in\{1,...,p\}$ and $(i_{k},i_{k+1})\in\mathcal{E}$ for $k\in\{1,...,p-1\}$.
Under the assumption that the path does not meet the same vertex twice, a path satisfying $i_1 = i_p$ is called \textit{cycle} and a path satisfying $i_1\in\mathcal{V}_\mathcal{L}$ is called a \textit{stem}.
A graph that contains no cycles is called a \textit{directed acyclic graph (DAG)}.
We say that $\mathcal{G}$ is influenceable, if there exists a path from at least one leader to any state node.
Since state nodes without a path from at least one leader are uncontrollable, this paper only considers the graphs that are influenceable.

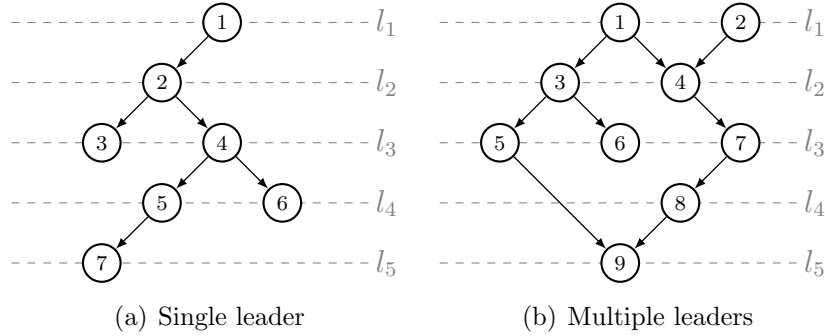
\begin{figure}[t]
\centering
\subfigure[Single leader]{
\begin{tikzpicture}[scale=0.8]
\draw[dashed, gray] (-1,4) -- (5,4) node[left, pos=1.1] { $l_1$};
\draw[dashed, gray] (-1,3) -- (5,3) node[left, pos=1.1] { $l_2$};
\draw[dashed, gray] (-1,2) -- (5,2) node[left, pos=1.1] { $l_3$};
\draw[dashed, gray] (-1,1) -- (5,1) node[left, pos=1.1] { $l_4$};
\draw[dashed, gray] (-1,0) -- (5,0) node[left, pos=1.1] { $l_5$};

\node[place, circle,minimum size=0.5cm] (node1) at (2.5,4) [] {\scriptsize$1$};
\node[place, circle,minimum size=0.5cm] (node2) at (1.5,3) [] {\scriptsize$2$};
\node[place, circle,minimum size=0.5cm] (node3) at (0.5,2) [] {\scriptsize$3$};
\node[place, circle,minimum size=0.5cm] (node4) at (2.5,2) [] {\scriptsize$4$};
\node[place, circle,minimum size=0.5cm] (node5) at (1.5,1) [] {\scriptsize$5$};
\node[place, circle,minimum size=0.5cm] (node6) at (3.5,1) [] {\scriptsize$6$};
\node[place, circle,minimum size=0.5cm] (node7) at (0.5,0) [] {\scriptsize$7$};

\draw (node1) [-latex, line width=0.5pt] -- node [right]  {} (node2);
\draw (node2) [-latex, line width=0.5pt] -- node [right]  {} (node3);
\draw (node2) [-latex, line width=0.5pt] -- node [right]  {} (node4);
\draw (node4) [-latex, line width=0.5pt] -- node [right]  {} (node5);
\draw (node4) [-latex, line width=0.5pt] -- node [right]  {} (node6);
\draw (node5) [-latex, line width=0.5pt] -- node [right]  {} (node7);
\end{tikzpicture}
}
\subfigure[Multiple leaders]{
\begin{tikzpicture}[scale=0.8]

\draw[dashed, gray] (-1,4) -- (5,4) node[left, pos=1.1] { $l_1$};
\draw[dashed, gray] (-1,3) -- (5,3) node[left, pos=1.1] { $l_2$};
\draw[dashed, gray] (-1,2) -- (5,2) node[left, pos=1.1] { $l_3$};
\draw[dashed, gray] (-1,1) -- (5,1) node[left, pos=1.1] { $l_4$};
\draw[dashed, gray] (-1,0) -- (5,0) node[left, pos=1.1] { $l_5$};

\node[place, circle,minimum size=0.5cm] (node1) at (2,4) [] {\scriptsize$1$};
\node[place, circle,minimum size=0.5cm] (node2) at (4,4) [] {\scriptsize$2$};
\node[place, circle,minimum size=0.5cm] (node3) at (1,3) [] {\scriptsize$3$};
\node[place, circle,minimum size=0.5cm] (node4) at (3,3) [] {\scriptsize$4$};
\node[place, circle,minimum size=0.5cm] (node5) at (0,2) [] {\scriptsize$5$};
\node[place, circle,minimum size=0.5cm] (node6) at (2,2) [] {\scriptsize$6$};
\node[place, circle,minimum size=0.5cm] (node7) at (4,2) [] {\scriptsize$7$};
\node[place, circle,minimum size=0.5cm] (node8) at (3,1) [] {\scriptsize$8$};
\node[place, circle,minimum size=0.5cm] (node9) at (2,0) [] {\scriptsize$9$};

\draw (node1) [-latex, line width=0.5pt] -- node [right]  {} (node3);
\draw (node1) [-latex, line width=0.5pt] -- node [right]  {} (node4);
\draw (node2) [-latex, line width=0.5pt] -- node [right]  {} (node4);
\draw (node3) [-latex, line width=0.5pt] -- node [right]  {} (node5);
\draw (node3) [-latex, line width=0.5pt] -- node [right]  {} (node6);
\draw (node4) [-latex, line width=0.5pt] -- node [right]  {} (node7);
\draw (node5) [-latex, line width=0.5pt] -- node [right]  {} (node9);
\draw (node7) [-latex, line width=0.5pt] -- node [right]  {} (node8);
\draw (node8) [-latex, line width=0.5pt] -- node [right]  {} (node9);

\end{tikzpicture}
}
\caption{Examples of $5$-layered DAGs. (a) The set of fixed nodes is $\{1,2,7\}$ with a single leader $1\in\mathcal{V}_{\mathcal{L}}$.
(b) The set of fixed nodes is $\{1,2,3,4,7,8,9\}$ with two leaders $1,2\in\mathcal{V}_{\mathcal{L}}$.}
\label{fig1}
\end{figure}

\subsection{The controllable subspace : graph characterization}
From a graph point of view, the generic dimension of controllable subspace for DAGs is given by:
\begin{theorem}\label{thm_hosoe}
 \cite{hosoe1980determination}
For a graph $\mathcal{G}$, 
the generic dimension of controllable subspace 
is the maximum number of state nodes that can be covered by a disjoint set of stems in $\mathcal{G}$.
\end{theorem}

Note that the maximum number of state nodes that can be covered by disjoint sets in a graph $\mathcal{G}$ is unique, while the disjoint set of stems in $\mathcal{G}$ may not be unique. 
For this reason, the authors in \cite{commault2017fixed} introduced a graph characterization of the fixed controllable subspace from an individual node perspective. 
\textcolor{black}{Specifically, in a graph $\mathcal{G}$, represented by the structured network \eqref{eq:2_1},
a state node $k\in\mathcal{V}$ is referred to as a \textit{fixed node} if its corresponding state $x_k$ is a controllable state.
In other words, the corresponding standard basis $e_k\in\mathbb{R}^{n}$ is included in the fixed controllable subspace.}

\begin{theorem}\label{thm_commault}
\cite{commault2017fixed}
Let us consider a state node $k\in\mathcal{V}$ in a graph $\mathcal{G}$.
Then, $k$ is a fixed node if and only if
$k$ becoming a leader with an additional input 
does not increase the generic dimension of controllable subspace of $\mathcal{G}$.
\end{theorem}

Based on \textit{Theorem~\ref{thm_hosoe}}, the above theorem provides a method to determine the fixed nodes in a graph.

\begin{example}
Let us consider a DAG $\mathcal{G}(\mathcal{V},\mathcal{E})$ with a single leader as shown in Fig.~\ref{fig1}(a).
The stem starting at $1\in\mathcal{V}_{\mathcal{L}}$ that covers the maximum number of state nodes is $(1\rightarrow2\rightarrow4\rightarrow5\rightarrow7)$.
From \textit{Theorem~\ref{thm_hosoe}}, the generic dimension of controllable subspace is 5.
Now, let us consider the state node $4$ becoming a leader with additional input, i.e., $4\in\mathcal{V}_{\mathcal{L}}$. 
Then, the set of two disjoint stems starting at $1,4\in\mathcal{V}_{\mathcal{L}}$ that covers the maximum number of state nodes is $\{(1\rightarrow2\rightarrow3),(4\rightarrow5\rightarrow7)\}$. 
In this case, the number of state nodes in the two disjoint stems is 6, which is greater than the generic dimension of controllable subspace without additional input. 
According to \textit{Theorem~\ref{thm_commault}}, the state node $4\in\mathcal{V}$ is not a fixed node.
On the contrary, if $2$ or $7$ becomes a leader, the generic dimension of controllable subspace does not increase, and thus, the state nodes $2$ and $7$ are considered to be fixed nodes. 
Note that since connecting multiple inputs to a state node does not increase the generic dimension of controllable subspace,
the leaders with existing inputs are trivially considered fixed nodes.
\end{example}

However, the method of determining fixed nodes using the conditions stated in \textit{Theorem~\ref{thm_commault}} requires a brute-force search, which can be computationally expensive. 
To address this issue, an alternative approach for determining fixed nodes in dynamic graphs has been proposed in \cite{van2019dynamic}, using the concept of \textit{supremal minimal separators}. 
Nevertheless, since the necessary and sufficient condition for fixed nodes presented in \cite{van2019dynamic} remains a conjecture, it continues to be an open problem. 
In summary, characterizing controllable subspaces and fixed nodes in structured networks can be achieved through both state-space and graph-theoretic perspectives. 
While existing literature has made progress, further research is needed to devise efficient methods for determining fixed nodes in graphs. 
In the following section, we characterize conditions for determining fixed nodes in DAGs.

\section{The Fixed Nodes in DAGs}\label{sec_fix}
Based on the necessary and sufficient condition for fixed nodes introduced in \textit{Theorem~\ref{thm_commault}}, 
this section provides several conditions of the fixed node in DAGs based on the notion of maximum disjoint stems.
\textcolor{black}{From \cite{liu2012control,della2020symmetries,yeh1998revised,bianchini2001processing},
it is well-known that every DAG can be represented as a unique hierarchical structure with layers, which simplifies the analysis of network controllability.}
Each state node in a DAG can be assigned a unique layer index using the recursive labeling algorithm in \cite{liu2012control,yan2010comparing}.
Given a DAG with the set of leaders $\mathcal{V}_{\mathcal{L}}$, the modified algorithm from \cite{liu2012control,yan2010comparing} could consist of the following steps:

\begin{itemize}
\item [(1)] \textcolor{black}{State nodes without incoming edges are labeled with layer index 1 (top layer), i.e., the highest level of influence.}
\item [(2)] \textcolor{black}{Remove all state nodes in layer 1, and the state nodes in the remaining graph without incoming edges are labeled with layer index 2.}
\item [(3)] \textcolor{black}{Repeat step (2) until all state nodes are labeled.}
\end{itemize}

This algorithm allows us to represent any DAG as a $p$-layered hierarchical structure, which we call a $p$-layered DAG.
In a $p$-layered DAG, state nodes in the same layer do not have edges with each other, ensuring the separation of influence levels,
and state nodes in a layer can only have edges toward state nodes in lower layers (downstream), maintaining the directed acyclic property.
It also means that the edge of the $p$-layered DAG may be connected between layers that are not adjacent to each other.
The layers of a $p$-layered DAG are denoted as $l_k$ for $k\in\{1,...,p\}$, and the set of state nodes in each layer is denoted as $\mathcal{V}_{l_k}$.
Note that from the initial step (1), all leaders in $\mathcal{V}_{\mathcal{L}}$ can only exist in $\mathcal{V}_{l_1}$, i.e., $\mathcal{V}_{\mathcal{L}}\subseteq\mathcal{V}_{l_1}$.

\subsection{Single leader}
For the case of a single leader, consider the following necessary and sufficient conditions for fixed nodes in DAGs.

\begin{theorem}\label{thm1}
For a $p$-layered DAG $\mathcal{G}(\mathcal{V},\mathcal{E})$ with a single leader $j\in\mathcal{V}_{\mathcal{L}}$, 
the state node $i\in \mathcal{V}_{l_k}$ is a fixed node if and only if $|\mathcal{V}_{l_k}|=1$ for $k\in\{1,...,p\}$.
\end{theorem}
\begin{proof}
Let $\mathcal{G}_S^{max}$ be the stem that covers the maximum number of state nodes in $\mathcal{G}$.
Note that for a $p$-layered DAG with a single leader, the maximum number of state nodes covered by a stem is equal to the maximum index of layer $p$.
For the \textit{if} condition, suppose that a state node $i\in \mathcal{V}_{l_k}$ is the only node in its layer, i.e., $|\mathcal{V}_{l_k}|=1$. 
Then, the length of the stem from the leader $j$ to $i$ is $k$.
Furthermore, any stem with length $k$ or greater always includes $i$ in its sequence.
 Hence, even if $i$ becomes an additional leader, 
the maximum number of state nodes covered by two stems (starting from $i$ and $j$) is still at most the number of state nodes covered by the maximum disjoint stem $\mathcal{G}_S^{max}$.

For the \textit{only if} condition, suppose that $i\in \mathcal{V}_{l_k}$ is a fixed node with $|\mathcal{V}_{l_k}|>1$.
Then, from \textit{Theorem~\ref{thm_commault}}, the maximum number of state nodes covered by stems should not increase even if $i$ becomes an additional leader.
To show that this assumption is contradictory, consider the following two cases:
First, if $i$ is not included in $\mathcal{G}_S^{max}$, 
then adding $i$ as an additional leader increases the generic dimension of controllable subspace since the new stem starting from $i$ covers at least one state node that was not covered before. 
Second, if $i$ is included in $\mathcal{G}_S^{max}$, then the stem starting from $i$ can always be included in $\mathcal{G}_S^{max}$, 
and the stem starting from $j$ can cover at least one more state node in $\mathcal{V}_{l_k}$. 
This contradicts the condition of fixed nodes stated in \textit{Theorem~\ref{thm_commault}}. 
Therefore, the fixed node must satisfy $|\mathcal{V}_{l_k}|=1$.
\end{proof}

The above theorem provides {black} for a state node to be a fixed node in a $p$-layered DAG with a single leader.
In other words, a state node is a fixed node if and only if it is the only node in its layer.

\begin{example}
Let us consider the $5$-layered DAG with a leader $1\in\mathcal{V}_{\mathcal{L}}$ shown in Fig.~\ref{fig1}(a). 
In this case, the fixed nodes are $1$, $2$, and $7$, which are the only state nodes in their respective layers.
However, the state nodes $3,4\in\mathcal{V}_{l_2}$ and $5,6\in\mathcal{V}_{l_3}$ are not fixed nodes because there are other nodes in their layers.
It is clear that the dimension of controllable subspace increases when non-fixed nodes become leaders.
\end{example}

\subsection{Multiple leaders}
For the case of multiple leaders, we can derive the following sufficient condition for fixed nodes using the same logic as the proof of the \textit{if} condition in \textit{Theorem~\ref{thm1}}:

\begin{corollary}\label{col1}
For a $p$-layered DAG $\mathcal{G}(\mathcal{V},\mathcal{E})$ with $m$-leader, 
the state node $i\in \mathcal{V}_{l_k}$ is a fixed node if there exists exactly one state node in $\mathcal{V}_{l_k}$, i.e., $|\mathcal{V}_{l_k}|=1$ for $k\in\{1,...,p\}$.
\end{corollary}

With the above corollary, let us consider a state node $j\in\mathcal{V}_{l_k}$ with $|\mathcal{V}_{l_k}|=1$.
If there exists a stem from a leader to $j$, the state node $j$ is obviously a fixed node.
In this case, the subgraph of $\mathcal{G}(\mathcal{V},\mathcal{E})$ induced by the set of vertices $\bigcup_{i=k}^{p}\mathcal{V}_{l_i}$
can be interpreted as a subgraph with a single leader $j$.
In this case, for the subgraph induced by $\bigcup_{i=k}^{p}\mathcal{V}_{l_i}$, we can use the necessary and sufficient condition for a single leader provided in \textit{Theorem~\ref{thm1}}.
On the other hand, if there does not exist a stem from a leader to $j$, the subgraph induced by the set of vertices $\bigcup_{i=k}^{p}\mathcal{V}_{l_i}$ can also be interpreted as a subgraph without a leader.
As a result, there are only non-fixed nodes in $\bigcup_{i=k}^{p}\mathcal{V}_{l_i}$.
Now, consider the following sufficient condition for non-fixed nodes:

\begin{lemma}\label{lem1}
For a DAG $\mathcal{G}(\mathcal{V},\mathcal{E})$, let $\mathcal{G}_D(\mathcal{V}_D,\mathcal{E}_D)$ be an union graph of disjoint stems 
that cover the maximum number of state nodes in $\mathcal{G}$.
Then, the state node $k\in\mathcal{V}$ is not a fixed node if $k\in(\mathcal{V}\setminus\mathcal{V}_D)$.
\end{lemma}
\begin{proof}
The proof is straightforward and based on the condition presented in \textit{Theorem~\ref{thm_commault}}. 
From \textit{Theorem~\ref{thm_hosoe}}, the number of state nodes in $\mathcal{V}_D$ is equal to the generic dimension of controllable subspace. 
Hence, if $k\in(\mathcal{V}\setminus\mathcal{V}_D)$ and becomes a leader, i.e., $k\in\mathcal{V}_{\mathcal{L}}$, the dimension of controllable subspace increases by at least one.
\end{proof}

\begin{example}
As an example of the above lemma, let us consider the graph $\mathcal{G}(\mathcal{V},\mathcal{E})$ in Fig.~\ref{fig1}(b).
In this case, the set of two disjoint stems that covers the maximum number of state nodes in $\mathcal{G}$ can be
$\{(1\rightarrow3\rightarrow5\rightarrow9),(2\rightarrow4\rightarrow7\rightarrow8)\}$, which is one of the possible sets of disjoint stems.
Thus, we obtain $\mathcal{V}_D=\{1,2,3,4,5,7,8,9\}$. 
From \textit{Lemma~\ref{lem1}}, the state node $6\in\mathcal{V}\setminus\mathcal{V}_D$ is not a fixed node.
Similarly, if the set of disjoint stems that covers the maximum number of state nodes is given by 
$\{(1\rightarrow3\rightarrow6),(2\rightarrow4\rightarrow7\rightarrow8\rightarrow9)\}$, we obtain $\mathcal{V}_D=\{1,2,3,4,6,7,8,9\}$.
Then, the state node $5\in\mathcal{V}\setminus\mathcal{V}_D$ is not a fixed node.
It is clear that whenever a non-fixed node becomes a leader, the dimension of controllable subspace always increases.
\end{example}

Here we introduce several terms for analyzing DAGs from each layer's perspective. 
Let us consider a $p$-layered DAG $\mathcal{G}(\mathcal{V},\mathcal{E})$ with the set of $m$-leaders $\mathcal{V}_{\mathcal{L}}$. 
For each layer $l_k,k\in\{1,...,p\}$, a \textit{maximum disjoint stems} of $\mathcal{V}_{l_k}$ is the set of $m$-disjoint stems that covers maximum number of state nodes in $\mathcal{V}_{l_k}$.
The maximum disjoint stems of $\mathcal{V}_{l_k}$ is denoted by $\mathcal{M}(\mathcal{V}_{l_k})$ for $k\in\{1,...,p\}$, which is not necessarily unique.
A state node $i\in\mathcal{V}_{l_k}$ is called a \textit{matched node} if there exists a stem in $\mathcal{M}(\mathcal{V}_{l_k})$ that starts from a leader to $i$.
Finally, in each layer, a set of matched nodes of $\mathcal{M}(\mathcal{V}_{l_k})$ is denoted by $\widehat{\mathcal{M}}(\mathcal{V}_{l_k})$,
which is determined by taking the intersection of $\mathcal{M}(\mathcal{V}_{l_k})$ and $\mathcal{V}_{l_k}$, 
i.e., 
\begin{equation*}
\widehat{\mathcal{M}}(\mathcal{V}_{l_k})=\mathcal{M}(\mathcal{V}_{l_k})\cap\mathcal{V}_{l_k}\,\, \text{for} \,\, k\in\{1,...,p\}.
\end{equation*}

An example of the terms defined above is given below.

\begin{example}
Let us consider the $4$-layered DAG shown in Fig.~\ref{fig1}(b) and suppose that $1,2\in\mathcal{V}_{\mathcal{L}}$.
In the first layer, the only feasible $\mathcal{M}(\mathcal{V}_{l_1})$ is $\{1,2\}$,
which are matched by themselves. Thus, the set of matched nodes in $\mathcal{V}_{l_1}$, i.e., $\widehat{\mathcal{M}}(\mathcal{V}_{l_1})$, is $\{1,2\}$.
Moving to the second layer, the set of disjoint stems that covers the maximum number of state nodes in $\mathcal{V}_{l_2}$,
i.e., $\mathcal{M}(\mathcal{V}_{l_2})$, is $\{(1\rightarrow3), (2\rightarrow4)\}$.
Hence, we obtain $\widehat{\mathcal{M}}(\mathcal{V}_{l_2})=\{3,4\}$.
Note that the set $\{4\}$ with the stem $1\rightarrow4$ cannot be a set of matched nodes, since this is not the \textit{maximum} set.
Now in the third layer, the number of feasible maximum disjoint stems of $\mathcal{V}_{l_3}$ is 2,
thus, $\mathcal{M}(\mathcal{V}_{l_3})$ is either
$\{(1\rightarrow3\rightarrow5), (2\rightarrow4\rightarrow7)\}$ or
$\{(1\rightarrow3\rightarrow6), (2\rightarrow4\rightarrow7)\}$.
Hence, $\widehat{\mathcal{M}}(\mathcal{V}_{l_3})$ can be either $\{5,7\}$ or $\{6,7\}$.
In the fourth layer, the feasible $\mathcal{M}(\mathcal{V}_{l_4})$ is either the stem $(1\rightarrow4\rightarrow7\rightarrow8)$ or $(2\rightarrow4\rightarrow7\rightarrow8)$,
it follows that $\widehat{\mathcal{M}}(\mathcal{V}_{l_4})$ is only $\{8\}$.
Similarly, in the fifth layer, we obtain $\widehat{\mathcal{M}}(\mathcal{V}_{l_5})={9}$.
\end{example}

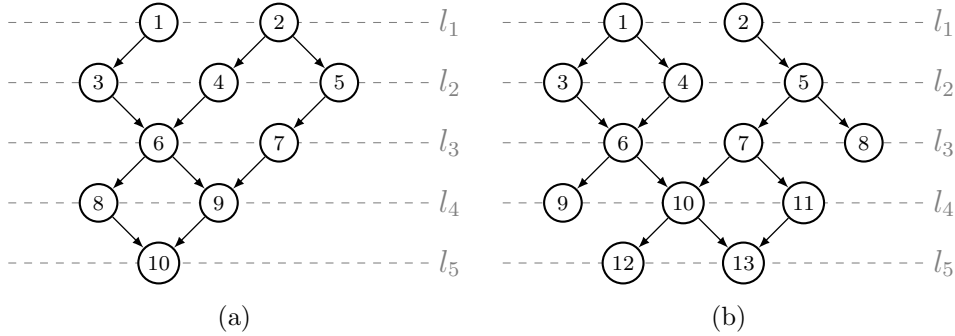
\begin{figure}[t]
\centering
\subfigure[]{
\begin{tikzpicture}[scale=0.8]
\draw[dashed, gray] (-1.5,4) -- (5.5,4) node[left, pos=1.1] { $l_1$};
\draw[dashed, gray] (-1.5,3) -- (5.5,3) node[left, pos=1.1] { $l_2$};
\draw[dashed, gray] (-1.5,2) -- (5.5,2) node[left, pos=1.1] { $l_3$};
\draw[dashed, gray] (-1.5,1) -- (5.5,1) node[left, pos=1.1] { $l_4$};
\draw[dashed, gray] (-1.5,0) -- (5.5,0) node[left, pos=1.1] { $l_5$};

\node[place, circle,minimum size=0.5cm] (node1) at (1,4) [] {\scriptsize$1$};
\node[place, circle,minimum size=0.5cm] (node2) at (3,4) [] {\scriptsize$2$};
\node[place, circle,minimum size=0.5cm] (node3) at (0,3) [] {\scriptsize$3$};
\node[place, circle,minimum size=0.5cm] (node4) at (2,3) [] {\scriptsize$4$};
\node[place, circle,minimum size=0.5cm] (node5) at (4,3) [] {\scriptsize$5$};
\node[place, circle,minimum size=0.5cm] (node6) at (1,2) [] {\scriptsize$6$};
\node[place, circle,minimum size=0.5cm] (node7) at (3,2) [] {\scriptsize$7$};
\node[place, circle,minimum size=0.5cm] (node8) at (0,1) [] {\scriptsize$8$};
\node[place, circle,minimum size=0.5cm] (node9) at (2,1) [] {\scriptsize$9$};
\node[place, circle,minimum size=0.5cm] (node10) at (1,0) [] {\scriptsize$10$};

\draw (node1) [-latex, line width=0.5pt] -- node [right]  {} (node3);
\draw (node2) [-latex, line width=0.5pt] -- node [right]  {} (node4);
\draw (node2) [-latex, line width=0.5pt] -- node [right]  {} (node5);
\draw (node3) [-latex, line width=0.5pt] -- node [right]  {} (node6);
\draw (node4) [-latex, line width=0.5pt] -- node [right]  {} (node6);
\draw (node5) [-latex, line width=0.5pt] -- node [right]  {} (node7);
\draw (node6) [-latex, line width=0.5pt] -- node [right]  {} (node8);
\draw (node6) [-latex, line width=0.5pt] -- node [right]  {} (node9);
\draw (node7) [-latex, line width=0.5pt] -- node [right]  {} (node9);
\draw (node8) [-latex, line width=0.5pt] -- node [right]  {} (node10);
\draw (node9) [-latex, line width=0.5pt] -- node [right]  {} (node10);

\end{tikzpicture}
}
\subfigure[]{
\begin{tikzpicture}[scale=0.8]

\draw[dashed, gray] (-1,4) -- (6,4) node[left, pos=1.1] { $l_1$};
\draw[dashed, gray] (-1,3) -- (6,3) node[left, pos=1.1] { $l_2$};
\draw[dashed, gray] (-1,2) -- (6,2) node[left, pos=1.1] { $l_3$};
\draw[dashed, gray] (-1,1) -- (6,1) node[left, pos=1.1] { $l_4$};
\draw[dashed, gray] (-1,0) -- (6,0) node[left, pos=1.1] { $l_5$};

\node[place, circle,minimum size=0.5cm] (node1) at (1,4) [] {\scriptsize$1$};
\node[place, circle,minimum size=0.5cm] (node2) at (3,4) [] {\scriptsize$2$};
\node[place, circle,minimum size=0.5cm] (node3) at (0,3) [] {\scriptsize$3$};
\node[place, circle,minimum size=0.5cm] (node4) at (2,3) [] {\scriptsize$4$};
\node[place, circle,minimum size=0.5cm] (node5) at (4,3) [] {\scriptsize$5$};
\node[place, circle,minimum size=0.5cm] (node6) at (1,2) [] {\scriptsize$6$};
\node[place, circle,minimum size=0.5cm] (node7) at (3,2) [] {\scriptsize$7$};
\node[place, circle,minimum size=0.5cm] (node8) at (5,2) [] {\scriptsize$8$};
\node[place, circle,minimum size=0.5cm] (node9) at (0,1) [] {\scriptsize$9$};
\node[place, circle,minimum size=0.5cm] (node10) at (2,1) [] {\scriptsize$10$};
\node[place, circle,minimum size=0.5cm] (node11) at (4,1) [] {\scriptsize$11$};
\node[place, circle,minimum size=0.5cm] (node12) at (1,0) [] {\scriptsize$12$};
\node[place, circle,minimum size=0.5cm] (node13) at (3,0) [] {\scriptsize$13$};

\draw (node1) [-latex, line width=0.5pt] -- node [right]  {} (node3);
\draw (node1) [-latex, line width=0.5pt] -- node [right]  {} (node4);
\draw (node2) [-latex, line width=0.5pt] -- node [right]  {} (node5);
\draw (node3) [-latex, line width=0.5pt] -- node [right]  {} (node6);
\draw (node4) [-latex, line width=0.5pt] -- node [right]  {} (node6);
\draw (node5) [-latex, line width=0.5pt] -- node [right]  {} (node7);
\draw (node5) [-latex, line width=0.5pt] -- node [right]  {} (node8);
\draw (node6) [-latex, line width=0.5pt] -- node [right]  {} (node9);
\draw (node6) [-latex, line width=0.5pt] -- node [right]  {} (node10);
\draw (node7) [-latex, line width=0.5pt] -- node [right]  {} (node10);
\draw (node7) [-latex, line width=0.5pt] -- node [right]  {} (node11);
\draw (node10) [-latex, line width=0.5pt] -- node [right]  {} (node12);
\draw (node10) [-latex, line width=0.5pt] -- node [right]  {} (node13);
\draw (node11) [-latex, line width=0.5pt] -- node [right]  {} (node13);

\end{tikzpicture}
}
\caption{Examples of $5$-layered DAGs. (a) The set of fixed nodes is $\{1,2,3,6,7,8,9,10\}$ with leaders $1,2\in\mathcal{V}_{\mathcal{L}}$.
(b) The set of fixed nodes is $\{1,2,5,6,12,13\}$ with leaders $1,2\in\mathcal{V}_{\mathcal{L}}$.}
\label{fig2}
\end{figure}

Consider the following sufficient condition of fixed nodes.

\begin{theorem}\label{thm2}
For a $p$-layered DAG $\mathcal{G}(\mathcal{V},\mathcal{E})$ with the set of $m$-leaders $\mathcal{V}_{\mathcal{L}}$, 
the state nodes in $\mathcal{V}_{l_k}$ are fixed nodes if 
there exists exactly one feasible set of matched nodes in $\mathcal{V}_{l_k}$ for $k\in\{1,...,p\}$.
\end{theorem}
\begin{proof}
For a $p$-layered DAG $\mathcal{G}(\mathcal{V},\mathcal{E})$ with $m$-leaders, 
let us suppose that there exists exactly one feasible set of matched nodes in a layer $\mathcal{V}_{l_j}$, where $j \in \{1,..., p\}$.
Now, let us consider stems in $\mathcal{G}$ that start at a leader in $\mathcal{V}_{\mathcal{L}}$ to a state node in $\mathcal{V}_{l_k},j < k \leq p$. 
Since there is only one feasible set of matched nodes in layer $\mathcal{V}_{l_j}$, 
every stem must include a state node in $\mathcal{V}_{l_j}$ that is matched by exactly one leader at the same time (one-to-one). 
This follows from the assumption that there is only one feasible set of matched nodes in $\mathcal{V}_{l_j}$.
Next, let us suppose that an additional input is added to a matched node $w \in \mathcal{V}_{l_j}$, i.e., $w \in \mathcal{V}_{\mathcal{L}}$. 
In this case, any newly generated stems by $w$ 
can always be included in one of the feasible stems starting at the existing leaders in $\mathcal{V}_{\mathcal{L}} \setminus \{w\}$. 
Therefore, even if a matched node in $\mathcal{V}_{l_k}$ becomes a leader, 
the maximum number of state nodes covered by disjoint stems does not increase.
From \textit{Theorem~\ref{thm_commault}}, it follows that the matched nodes in $\mathcal{V}_{l_k}$ are fixed nodes.
\end{proof}

\begin{example}
Let us consider the graph shown in Fig.~\ref{fig2}(a) and suppose that $1,2\in\mathcal{V}_{\mathcal{L}}$.
In this case, there exists only one feasible set of matched nodes in $\mathcal{V}_{l_1}$, $\mathcal{V}_{l_3}$, $\mathcal{V}_{l_4}$, and $\mathcal{V}_{l_5}$.
Specifically, in the second layer, the feasible set of matched nodes in $\mathcal{V}_{l_2}$ can either be $\{3,4\}$ or $\{3,5\}$ 
with the sets of disjoint stems $\{(1\rightarrow3),(2\rightarrow4)\}$ and $\{(1\rightarrow3),(2\rightarrow5)\}$, respectively.
In this case, there may exist a state node in $\mathcal{V}_{l_2}$ that is not a fixed node.
On the other hand, in the third layer, every stem with a length of $3$ or more must include one of the state nodes $6,7\in\mathcal{V}_{l_3}$.
Therefore, even if a state node in $\mathcal{V}_{l_3}$ becomes a leader, it is clear that the maximum number of state nodes that can be covered by disjoint stems cannot increase.
The same logic can be applied to $\mathcal{V}_{l_1}$, $\mathcal{V}_{l_4}$, and $\mathcal{V}_{l_5}$.
\end{example}

For a more detailed analysis based on the maximum disjoint stems of each layer,
let all possible feasible maximum disjoint stems of $\mathcal{V}_{l_k}$ be denoted as $\mathcal{M}_{j}(\mathcal{V}_{l_k})$ for $k\in\{1,...,p\}$ and $j\in\{1,...,n_k\}$,
where $n_k$ is the number of feasible maximum disjoint stems of $\mathcal{V}_{l_k}$.
Similarly, the set of matched nodes is denoted as $\widehat{\mathcal{M}}_{j}(\mathcal{V}_{l_k})$,
which is defined by the intersection of the set of maximum disjoint stems and state nodes in each layer, 
i.e., 
\begin{equation*}
\widehat{\mathcal{M}}_{j}(\mathcal{V}_{l_k})=\mathcal{M}_{j}(\mathcal{V}_{l_k})\cap\mathcal{V}_{l_k}\,\, \text{for} \,\, k\in\{1,...,p\}.
\end{equation*}

With the above terms, the following theorem provides the necessary and sufficient condition for fixed nodes in DAGs.

\begin{theorem}\label{thm3}
For a $p$-layered DAG $\mathcal{G}(\mathcal{V},\mathcal{E})$ with the set of $m$-leaders $\mathcal{V}_{\mathcal{L}}$, 
the state node $i\in \mathcal{V}_{l_k}$ is fixed node if and only if
$i\in\bigcap_{j=1}^{n_k}\widehat{\mathcal{M}}_j(\mathcal{V}_{l_k})$ for $k\in\{1,...,p\}$.
\end{theorem}
\begin{proof}
For the \textit{if} condition, let us suppose a state node $i \in \mathcal{V}_{l_k}$ satisfies $i \in \bigcap_{j=1}^{n_k} \widehat{\mathcal{M}}_j(\mathcal{V}_{l_k})$ for $k \in \{1,...,p\}$. 
Then, $i$ is a matched node for all $\widehat{\mathcal{M}}_j(\mathcal{V}_{l_k})$, where $j \in \{1,...,n_k\}$. It follows that $i$ is in the intersection of all matched nodes. 
Consequently, even if $i$ becomes a leader, the state nodes that can be covered by the stem generated by $i$ 
can always be included in the original set of maximum disjoint stems by the existing leaders in $\mathcal{V}_{\mathcal{L}}$. 
Therefore, according to \textit{Theorem \ref{thm_commault}}, $i$ is a fixed node.
For the \textit{only if} condition, let us suppose $i$ is a fixed node but $i \notin \bigcap_{j=1}^{n_k} \widehat{\mathcal{M}}_j(\mathcal{V}_{l_k})$ for some $k \in \{1,...,p\}$. 
This implies there exists a set of matched nodes of $\mathcal{V}_{l_k}$ where $i$ is not matched. 
It follows that there exists at least one set of maximum disjoint stems of $\mathcal{V}_{l_k}$ that does not include $i$. 
In this case, if $i$ becomes a leader, it is clear that the maximum number of state nodes in $\mathcal{V}_{l_k}$ covered by disjoint stems should increase by at least one. 
However, this contradicts the assumption that $i$ is a fixed node according to \textit{Theorem \ref{thm_commault}}. 
Therefore, $i$ must be in the intersection of all matched nodes in $\mathcal{V}_{l_k}$ for $k \in \{1,...,p\}$.
\end{proof}

\begin{remark}
\textcolor{black}{
The authors in \cite{van2019dynamic} introduce the concept of the \textit{supremal minimal separator} to propose conditions for identifying non-fixed nodes. 
While they establish necessary conditions for fixed nodes, as detailed in \textit{Conjecture~1} of their paper, they do not provide a necessary and sufficient condition for fixed nodes. 
On the other hand, for DAGs, \textit{Theorem~\ref{thm3}} provides a necessary and sufficient condition for determining fixed nodes within $\mathcal{V}_{l_k}$ for each $k \in \{1, \ldots, p\}$. 
This distinction underscores that \textit{Theorem~\ref{thm3}} offers a significant contribution beyond the conditions presented in \cite{van2019dynamic} within DAGs.
}
\end{remark}

\textit{Theorem~\ref{thm3}} is crucial as it allows us to determine the fixed nodes in a structured network with a set of leaders.
Note that only a subgraph in $\mathcal{G}$ induced by the set of vertices $\bigcup_{w=1}^{k}\mathcal{V}_{l_w}$ is required to find the fixed nodes in $\mathcal{V}_{l_k}$.
This means that to find the fixed nodes in a specific layer, we do not need to consider the entire graph but only a smaller, more manageable subgraph.
Naturally, if $k = p$, the induced graph by the set of vertices $\bigcup_{w=1}^{p}\mathcal{V}_{l_w}$ becomes identical to the original graph $\mathcal{G}$. 
The time complexity of this theorem will be discussed later.
From the proof of \textit{Theorem~\ref{thm3}}, we can easily obtain the following sufficient condition for fixed nodes:

\begin{corollary}\label{cor2}
For a $p$-layered DAG $\mathcal{G}(\mathcal{V},\mathcal{E})$ with the set of $m$-leaders $\mathcal{V}_{\mathcal{L}}$, a state node $i\in \mathcal{V}_{l_k}$ is a fixed node if there exists exactly one incoming edge $(j,i)\in\mathcal{E}$ from a fixed node $j\in\mathcal{V}_{l_{k-1}}$, where $k\in\{2,...,p\}$.
\end{corollary}

\begin{example}
Let us consider the graph $\mathcal{G}(\mathcal{V},\mathcal{E})$ shown in Fig.~\ref{fig2}(b) and suppose that $1,2\in\mathcal{V}_{\mathcal{L}}$.
This example demonstrates the application of \textit{Theorem~\ref{thm3}} in identifying fixed nodes in a $5$-layered DAG. 
By analyzing each layer of the graph, we can determine the fixed nodes based on the set of matched nodes.
In the first layer, the leaders $1$ and $2$ are inherently fixed nodes, as they each have their own external input.
In the second layer, the feasible sets of matched nodes are:
\begin{align*}
\widehat{\mathcal{M}}_{1}(\mathcal{V}_{l_2})&=\{3,5\},\\
\widehat{\mathcal{M}}_{2}(\mathcal{V}_{l_2})&=\{4,5\},
\end{align*}
where $n_2=2$. Thus, we have $\bigcap_{j=1}^{2}\widehat{\mathcal{M}}_j(\mathcal{V}_{l_2})=\{5\}$.
From \textit{Theorem~\ref{thm3}}, the state node $5\in\mathcal{V}_{l_2}$ is a fixed node.
Similarly, in the third layer, the feasible sets of matched nodes are:
\begin{align*}
\widehat{\mathcal{M}}_{1}(\mathcal{V}_{l_3})=\{6,7\},\\
\widehat{\mathcal{M}}_{2}(\mathcal{V}_{l_3})=\{6,8\},
\end{align*}
where $n_3=2$. From $\bigcap_{j=1}^{2}\widehat{\mathcal{M}}_j(\mathcal{V}_{l_3})=\{6\}$, the state node $6\in\mathcal{V}_{l_3}$ is a fixed node in this case.
On the other hand, in the fourth layer, the feasible sets of matched nodes are: 
\begin{align*}
\widehat{\mathcal{M}}_{1}(\mathcal{V}_{l_4})&=\{9,10\},\\
\widehat{\mathcal{M}}_{2}(\mathcal{V}_{l_4})&=\{10,11\},\\
\widehat{\mathcal{M}}_{3}(\mathcal{V}_{l_4})&=\{9,11\},
\end{align*}
where $n_4=3$. In this case, the intersection of all feasible sets of matched nodes is an empty set, which indicates that there are no fixed nodes in $\mathcal{V}_{l_4}$
In the last layer, the feasible set of matched nodes is only $\{12,13\}$, thus, the state nodes $12,13\in\mathcal{V}_{l_5}$ are the fixed nodes.
This shows that the fixed nodes can be easily identified when there is only one feasible set of matched nodes as introduced in \textit{Theorem~\ref{thm2}}.
Finally, we can obtain the set of fixed nodes in $\mathcal{G}(\mathcal{V},\mathcal{E})$ as $\{1,2,5,6,12,13\}$. 
This example illustrates the practical application of \textit{Theorem~\ref{thm3}} in a given network structure and provides a step-by-step approach to identifying fixed nodes.
\end{example}

\section{\textcolor{black}{Algorithm for Searching Fixed Nodes}}\label{sec_alg}
This section presents an efficient algorithm that simultaneously performs labeling and fixed node search for general DAGs.
In particular, \textit{Theorem~\ref{thm3}} shares similarities with the recursive labeling algorithm proposed in \cite{yan2010comparing,liu2012control}. 
The recursive labeling algorithm removes state nodes without incoming edges and labels them sequentially, from the top layer to the bottom layer, until all state nodes have been labeled. 
By repeating the algorithm $p$-times, the set of state nodes in each layer, denoted by $\mathcal{V}_{l_k}$ for $k\in\{1,...,p\}$, can be obtained. 
At each iteration $k$, a subgraph, denoted by $\mathcal{G}_{l_k}$ induced by the set $\bigcup_{w=1}^{k}\mathcal{V}_{l_w}$, can be constructed. 
Note that the induced subgraph $\mathcal{G}_{l_p}$ becomes the original graph $\mathcal{G}$ when $k=p$.
Similarly in \textit{Theorem~\ref{thm3}}, to determine whether a state node $i\in\mathcal{V}_{l_k}$ satisfies the condition of a fixed node, 
we only need to consider the induced subgraph $\mathcal{G}_{l_k}$ for finding $\widehat{\mathcal{M}}_j(\mathcal{V}_{l_k})$ for $j\in\{1,...,n_k\}$. 
Hence, by applying \textit{Theorem~\ref{thm3}} sequentially from the top layer to the bottom layer, 
the fixed node search and the recursive labeling algorithm can be performed simultaneously, resulting in a significant reduction in time complexity. 
This advantage is particularly useful for complex DAGs. 
The aforementioned process can be represented by \textit{Algorithm~\ref{alg}}, which applies to DAGs represented as $p$-layered hierarchical structure.

\begin{algorithm}[t]
\caption{Labeling-based fixed node search for DAGs}\label{alg}
\begin{algorithmic}[1]
\Procedure{FixedNodes}{a DAG $\mathcal{G}(\mathcal{V},\mathcal{E}),\mathcal{V}_{\mathcal{L}}$}
\State $k\gets1$ \Comment{Initialize the layer index}
\State $\mathcal{F}(\mathcal{V})\gets\emptyset$ \Comment{Initialize the set of fixed nodes}
\State $\overline{\mathcal{V}}\gets\mathcal{V}$ \Comment{For labeling}
\State $\overline{\mathcal{E}}\gets$ \textcolor{black}{Find edges from external inputs to each leader in $\mathcal{V}_{\mathcal{L}}$}
\State $\mathcal{E}\gets\mathcal{E}\setminus\overline{\mathcal{E}}$
\While{$\overline{\mathcal{V}}\neq\emptyset$}
\State $\mathcal{V}_{l_k}\gets i\in\overline{\mathcal{V}}$ if $\mathcal{N}_i\cap\overline{\mathcal{V}}=\emptyset$
\State Construct a graph $\mathcal{G}_{l_k}$ induced by $\bigcup_{w=1}^k \mathcal{V}_{l_k}$ 
\State Find feasible $\mathcal{M}_{j}(\mathcal{V}_{l_k})$ in $\mathcal{G}_{l_k}$ for $j\in\{1,...,n_k\}$
\For{$j\gets 1$ to $n_k$}
\State $\widehat{\mathcal{M}}_{j}(\mathcal{V}_{l_k})\gets\mathcal{M}_{j}(\mathcal{V}_{l_k})\cap\mathcal{V}_{l_k}$
\EndFor
\State $\mathcal{F}(\mathcal{V}_{l_k})\gets\bigcap_{j=1}^{n_k}\widehat{\mathcal{M}}_j(\mathcal{V}_{l_k})$
\State $\mathcal{F}(\mathcal{V})\gets\mathcal{F}(\mathcal{V})\cup\mathcal{F}(\mathcal{V}_{l_k})$
\State $\overline{\mathcal{V}}\gets\overline{\mathcal{V}}\setminus\mathcal{V}_{l_k}$
\State $k\gets k+1$
\EndWhile
\State $p\gets k$
\State \textbf{return} The set of fixed nodes $\mathcal{F}(\mathcal{V})$ in $\mathcal{G}(\mathcal{V},\mathcal{E})$
\EndProcedure \setcounter{ALG@line}{0}\newline
\end{algorithmic}
\end{algorithm}

\subsection{Computational complexity}
Now, we will discuss the time complexity of \textit{Theorem~\ref{thm3}} and \textit{Algorithm~\ref{alg}}. 
For a $p$-layered DAG $\mathcal{G}(\mathcal{V},\mathcal{E})$ with $\mathcal{V}_{\mathcal{L}}$, 
\textit{Theorem~\ref{thm3}} is based on the problem of finding $m$-disjoint stems starting at a leader in $\mathcal{V}_\mathcal{L}$ that covers the maximum number of state nodes in $\mathcal{V}_{l_k}$.
To solve this problem, let us consider two dummy nodes, $s$ and $t$.
By adding edges from $s$ to all leaders in $\mathcal{V}_\mathcal{L}$ and from all state nodes in $\mathcal{V}_{l_k}$ to $t$,
the problem of finding $m$-disjoint stems in \textit{Theorem~\ref{thm3}} can be transformed into a well-known \textit{$m$-vertex disjoint paths problem} ($m$-VDPP) \cite{robertson1995graph}.
Note that if $m$ is 1, the problem is equivalent to finding the shortest path in DAGs \cite{eppstein1998finding,takaoka1998shortest}, 
which has polynomial time complexity $\mathcal{O}(e+n)$, where $e$ is the number of edges and $n$ is the number of vertices.
For general directed graphs, the $m$-VDPP is known to be NP-complete even if $m=2$.
Fortunately, the authors in \cite{shiloach1978finding} showed that $2$-VDPP in DAGs can be solved in polynomial time with complexity $\mathcal{O}(en)$.
Furthermore, the authors in \cite{fortune1980directed} generalized the results of $2$-VDPP to 
$m$-VDPP in DAGs with complexity $\mathcal{O}(en^{m-1})$ for all $m\geq 2$.
It follows that if $m>2$, the $m$-VDPP in DAG becomes NP-hard with exponential time complexity.
Hence, for a DAG with $m$-leaders,
the time complexity of \textit{Theorem~\ref{thm3}} is $\mathcal{O}(e+n)$ if $m=1$, 
and $\mathcal{O}(en^{m-1})$ if $m\geq 2$.
Since the number of edges in a DAG is at most $n(n-1)/2$, which is the same as the number of edges in a complete undirected graph, 
we can substitute $e=n(n-1)/2$ into the time complexity expressions. 
Thus, we can also obtain the time complexity of \textit{Theorem~\ref{thm3}} as $\mathcal{O}(n^2)$ if $m=1$, and $\mathcal{O}(n^m)$ if $m\geq 2$.
On the other hand, in \textit{Algorithm~\ref{alg}}, the while loop runs for at most $n$ iterations because the maximum number of nodes in a layer is at most $n$.
Since the complexity of finding feasible maximum disjoint stems in a layer is $\mathcal{O}(n^m)$ if $m\geq 2$, 
the total time complexity of \textit{Algorithm~\ref{alg}} is $\mathcal{O}(n \cdot n^m) = \mathcal{O}(n^{m+1})$.
However, when there are multiple leaders ($m \geq 2$), the complexity increases, particularly for \textit{Algorithm~\ref{alg}}, which becomes $\mathcal{O}(n^{m+1})$. 
Nevertheless, the approach offers an intuitive perspective from a graph-theoretic standpoint, which aids in understanding the characteristics of DAGs, even when the complexity increases with multiple leaders. 
It is important to emphasize that these complexities are based on the worst-case scenario. 
In practice, the actual computational requirements may be significantly less depending on the structure and properties of a given DAG. 
Furthermore, there exist practical methods for reducing this complexity when applying the concepts to real DAGs, which will be introduced in the following section.

\begin{figure}[t] 
\centering 
\subfigure[Input : $6$-layered DAG $\mathcal{G}(\mathcal{V},\mathcal{E})$]{{\includegraphics[width=3.25in,height=1.3in,clip]{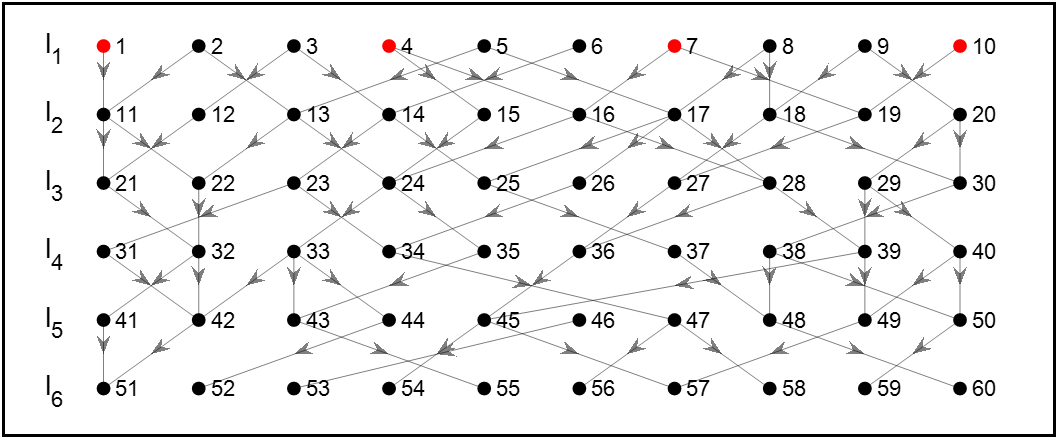}}}
    \subfigure[Output : fixed nodes in $\mathcal{G}(\mathcal{V},\mathcal{E})$]   {{\includegraphics[width=3.25in,height=1.3in,clip]{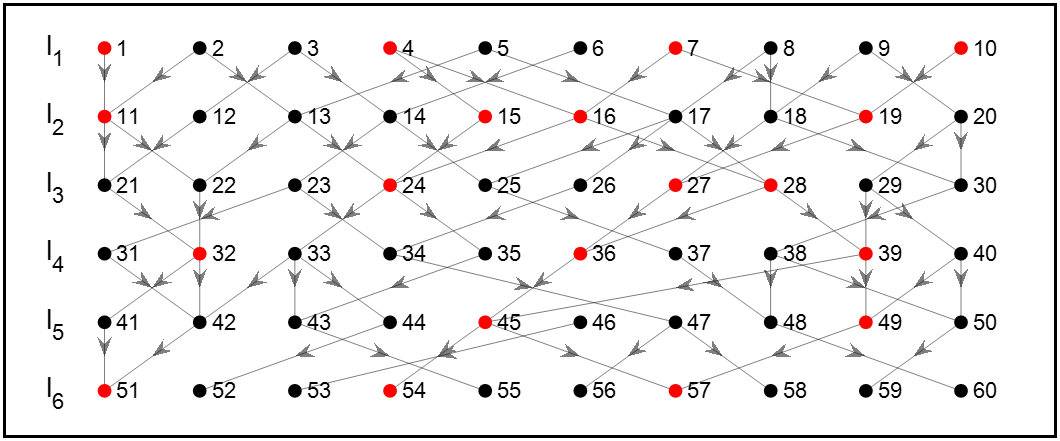}}}
    \subfigure[A set of disjoint stems in $\mathcal{G}(\mathcal{V},\mathcal{E})$]   {{\includegraphics[width=3.25in,height=1.3in,clip]{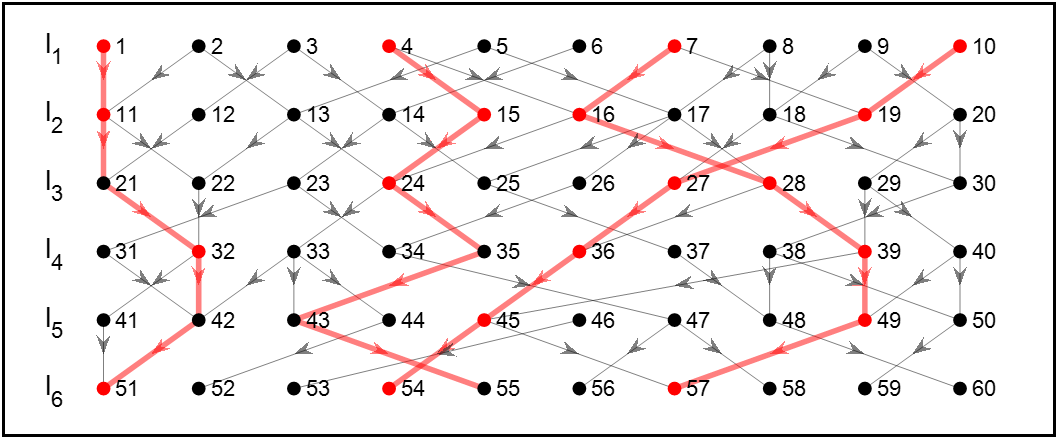}}}
\caption{Input and output graphs of \textit{Algorithm~\ref{alg}} implemented in MATLAB.
(a) $6$-layered DAG $\mathcal{G}(\mathcal{V},\mathcal{E})$ consists of $60$ state nodes ($10$ state nodes in each layer) and $80$ edges with $\mathcal{V}_{\mathcal{L}}=\{1,4,7,10\}$. The leaders (fixed nodes in the first layer) are shown in red.
(b) Output graph of \textit{Algorithm~\ref{alg}} with fixed nodes. The fixed nodes in $\mathcal{G}(\mathcal{V},\mathcal{E})$ are shown in red.
(c) One of the sets of disjoint stems that cover the maximum number of state nodes in $\mathcal{G}(\mathcal{V},\mathcal{E})$.
}  \label{fig3}
\end{figure}
\begin{figure}[t]
\centering
    \subfigure[Step 1]{{\includegraphics[width=3.2in,height=0.5in,clip]{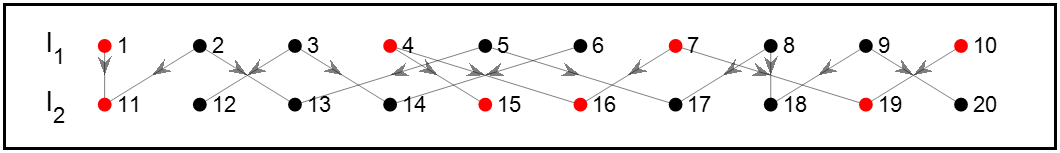}}}
    \subfigure[Step 2]{{\includegraphics[width=3.2in,height=0.7in,clip]{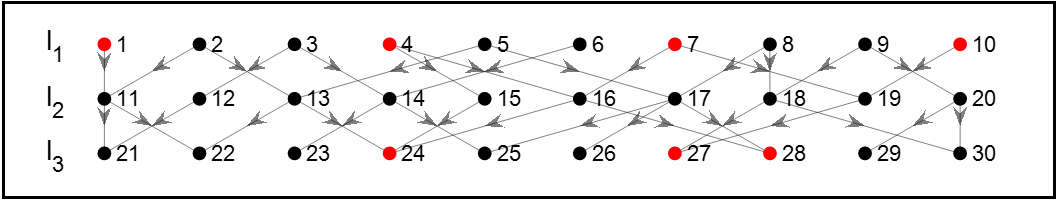}}}
    \subfigure[Step 3]{{\includegraphics[width=3.2in,height=0.9in,clip]{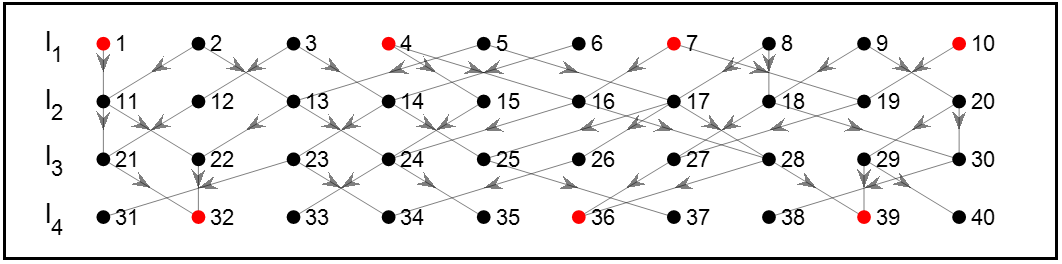}}}
    \subfigure[Step 4]{{\includegraphics[width=3.2in,height=1.1in,clip]{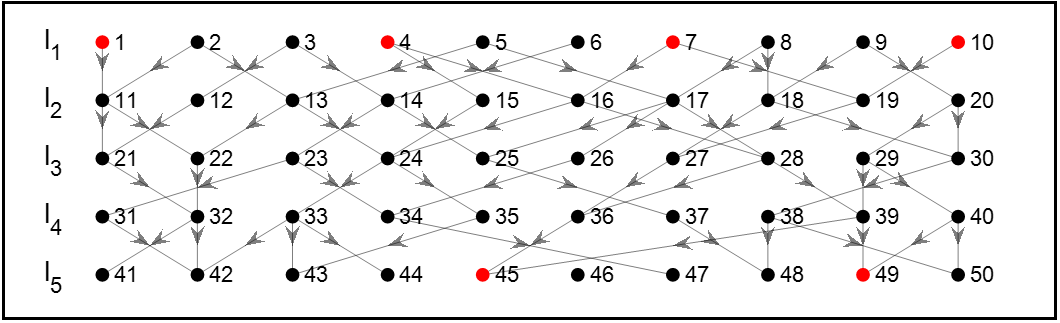}}}
    \subfigure[Step 5]{{\includegraphics[width=3.2in,height=1.4in,clip]{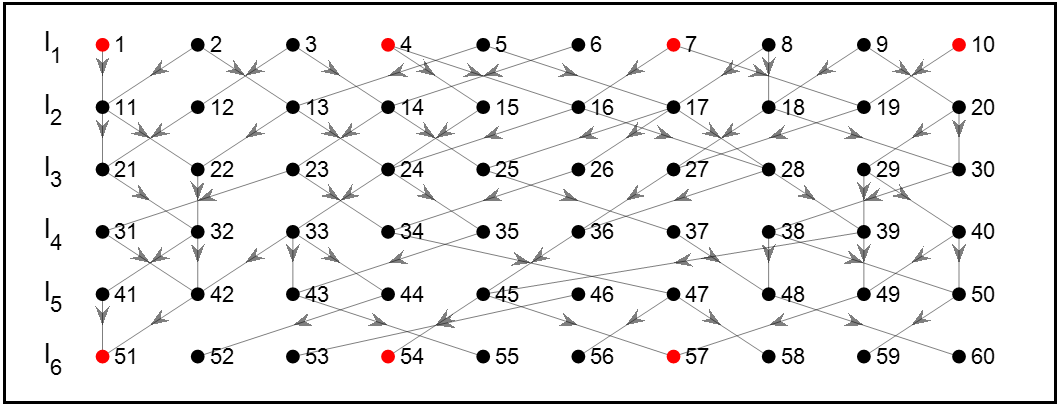}}}
    \caption{Intermediate results for each layer of \textit{Algorithm~\ref{alg}}. For intuition, the fixed nodes in each layer are shown in red.
    }
    \label{fig4}
\end{figure}

\section{Implementation of Algorithm}\label{sec_ex}
This section presents the implementation of \textit{Algorithm~\ref{alg}} in MATLAB, along with several techniques to reduce its time complexity.
For the $6$-layered DAG $\mathcal{G}(\mathcal{V},\mathcal{E})$ in Fig.~\ref{fig3}(a),
the simulation results of \textit{Algorithm~\ref{alg}} are depicted in Fig.~\ref{fig3}(b).
The detailed steps of each layer are illustrated in Fig.~\ref{fig4}.
The implementation of each step for layer $l_k, k \in \{1,...,6\}$ consists of the following steps:

\begin{itemize}
\item[(1)] Identify all possible stems to the state nodes in $\mathcal{V}_{l_k}$.
\item[(2)] Compute the sets of disjoint stems, which are formed by the combination of the stems starting from each leader.
\item[(3)] Select the sets containing the maximum number of state nodes in $\mathcal{V}_{l_k}$, i.e., the sets of matched nodes.
\item[(4)] Define the fixed nodes in $\mathcal{V}_{l_k}$ as the intersection of the sets of matched nodes in $\mathcal{V}_{l_k}$.
\end{itemize}

In \textit{Algorithm~\ref{alg}}, since finding feasible $\mathcal{M}_{j}(\mathcal{V}_{l_k})$ in $\mathcal{G}_{l_k}$ for $j \in \{1,...,n_k\}$ 
is based on the subgraph $\mathcal{G}_{l_k}$ induced by $\bigcup_{w=1}^k \mathcal{V}_{l_k}$, the computational requirement increases in proportion to the layer index.
However, during the process of (1), which calculates all stems starting from each leader to the target nodes (all state nodes in $\mathcal{V}_{l_k}$), 
excluding specific state nodes from the target nodes can significantly reduce the computational requirement. 
This is because calculating stems for fewer target nodes requires less time and resources.
To achieve this, several theories presented in this paper can be applied to reduce the target nodes in (1). 
For example, consider a set of disjoint stems that cover the maximum state nodes in $\mathcal{G}(\mathcal{V},\mathcal{E})$ as shown in Fig.~\ref{fig3}(c). 
By focusing on this set of stems, we can potentially eliminate unnecessary calculations for state nodes that are not likely to be fixed nodes.
From \textit{Lemma~\ref{lem1}}, state nodes not included in the set of disjoint stems that cover the maximum state nodes are always non-fixed nodes. 
Hence, only state nodes in the set of disjoint stems have the possibility to be fixed nodes. 
As a result, target nodes in step (1) can be reduced to only those nodes, which in turn reduces the overall computational requirement.

\begin{remark}
Note that since \textit{Theorem~\ref{thm_hosoe}} was proven to be solvable in polynomial time in \cite{poljak1990generic,commault2017fixed}, 
the problem of finding a set of disjoint stems that cover the maximum number of state nodes in DAGs also has polynomial complexity. 
This implies that even when these theories are employed to narrow down the target nodes, the overall time complexity remains manageable.
This not only enhances the efficiency by significantly reducing the algorithm's time complexity but also ensures the effectiveness of the algorithm is not compromised.
\end{remark}

\section{Conclusion}\label{sec_con}
In this paper, we have investigated the problem of identifying fixed nodes in DAGs represented as hierarchical structures with a unique label for each layer.
Our study focused on analyzing fixed nodes from both state-space and graph-theoretic perspectives and provided several properties of fixed nodes in DAGs. 
Furthermore, we introduced the concept of $m$-maximum disjoint stems to determine fixed nodes in DAGs and provided several conditions for identifying these nodes.
We have also discussed the computational complexity of our approach based on the number of leaders in DAGs. 
From a graph point of view, our study has contributed to the understanding of fixed nodes in $p$-layered DAGs by providing necessary and sufficient conditions for single and multiple leader cases, 
offering valuable insights into the controllability of individual nodes.
Despite the progress made in this paper, several open problems remain. Our proposed algorithm has a high computational cost, especially in large-scale networks, because it can be an NP-hard problem depending on the number of leaders (if $m\geq 2$). Therefore, our future work will focus on developing effective algorithms with lower time complexity and applying them to more complex graphs. 
In conclusion, this paper represents a significant step forward in understanding fixed nodes in structured networks. 


\section{Acknowledgment}
The work of this paper has been supported by the Institute of Planning and Evaluation for Technology (IPET) of Korea under the grant 1545026393.

\bibliographystyle{elsarticle-num}
\bibliography{ref_fixed} 

\end{document}